\documentclass[12pt]{amsart}
\usepackage{amsmath, amssymb, amsfonts}
\usepackage[all]{xypic}
\usepackage[pdftex]{graphicx}
\usepackage[usenames]{color}
\usepackage{stmaryrd}
\usepackage[utf8]{inputenc}
\usepackage{bbold}

\theoremstyle{plain}
\newtheorem{theorem}{Theorem}[section]
\newtheorem{lemma}[theorem]{Lemma}
\newtheorem{corollary}[theorem]{Corollary}
\newtheorem{prop}[theorem]{Proposition}

\newtheorem{conj}[theorem]{Conjecture}

\theoremstyle{remark}

\newtheorem{remark}[theorem]{Remark}

\newtheorem*{note*}{Note}
\newtheorem*{remark*}{Remark}
\newtheorem*{example*}{Example}

\theoremstyle{definition}
\newtheorem*{definition*}{Definition}
\newtheorem*{hypothesis*}{Hypothesis}
\newtheorem*{assumptions*}{Assumptions}
\newtheorem{definition}[theorem]{Definition}

\newcommand{\Z}{\mathbb{Z}}

\newcommand{\Q}{\mathbb{Q}}

\newcommand{\N}{\mathbb{N}}

\newcommand{\Aut}{\mathrm{Aut}}
\newcommand{\Gal}{\mathrm{Gal}}

\newcommand{\ab}{\mathrm{ab}}
\newcommand{\rad}{\mathrm{rad}}

\newcommand{\Hom}{\mathrm{Hom}}

\newcommand{\id}{\mathrm{id}}

\newcommand{\Mod}{\mathrm{Mod}}
\newcommand{\PMod}{\mathrm{PMod}}

\newcommand{\Ho}{\mathrm{Ho}}
\newcommand{\Ext}{\mathrm{Ext}}
\newcommand{\pd}{\mathrm{pd}}
\newcommand{\Ind}{\mathrm{Ind}}

\numberwithin{equation}{section}

\title[A generalization of a theorem of Swan and Iwasawa theory]
{A generalization of a theorem of Swan\\ 
	with applications to Iwasawa theory}

\author{Andreas Nickel}
\address{Universit\"{a}t Duisburg--Essen\\
	Fakult\"{a}t f\"{u}r Mathematik\\
	Thea-Leymann-Str. 9\\
	45127 Essen\\
	Germany}
\email{andreas.nickel@uni-due.de}
\urladdr{https://www.uni-due.de/$\sim$hm0251/english.html}

\date{Version of 26th March 2019}
\begin{document}

\begin{abstract}
	Let $p$ be a prime and let $G$ be a finite group. By a celebrated theorem 
	of Swan, two finitely generated projective $\Z_p[G]$-modules $P$ and $P'$
	are isomorphic if and only if $\Q_p \otimes_{\Z_p} P$ and
	$\Q_p \otimes_{\Z_p} P'$ are isomorphic as $\Q_p[G]$-modules.
	We prove an Iwasawa-theoretic analogue of this result
	and apply this to the Iwasawa theory of local and global fields.
	We thereby determine the structure of natural Iwasawa modules
	up to (pseudo-)isomorphism.
\end{abstract}

\maketitle

\section{Introduction}
Let $p$ be a prime and let $\mathcal{G}$ be a profinite group.
We denote the complete group algebra of $\mathcal{G}$ over $\Z_p$
by $\Lambda(\mathcal{G})$. In classical Iwasawa theory one studies
modules over $\Lambda(\Gamma)$ with $\Gamma \simeq \Z_p$
up to pseudo-isomorphism.
Jannsen \cite{MR1097615} has proposed a method for studying 
$\Lambda(\mathcal{G})$-modules up to isomorphism, which in fact works
for more general $\mathcal{G}$. 

In equivariant Iwasawa theory one is often concerned with the case where
$\mathcal{G}$ is a one-dimensional $p$-adic Lie group.
Then $\mathcal{G}$ may be written as a semi-direct product $H \rtimes \Gamma$
with a finite normal subgroup $H$ and $\Gamma \simeq \Z_p$.
Jannsen's theory works particularly nice if $\mathcal{G} =
H \times \Gamma$ is a direct product and $p$ does not divide
the cardinality of $H$ (see \cite[Chapter XI, \S 2 and \S 3]{MR2392026}).

As a concrete example, let $L/K$ be a finite Galois extension of
$p$-adic fields with Galois group $H$, where
a $p$-adic field shall always mean a finite extension of $\Q_p$
in this article.
Let $L_{\infty}$ be the
cyclotomic $\Z_p$-extension of $L$. 
We denote the $n$-th layer of the $\Z_p$-extension $L_{\infty}/L$
by $L_n$ as usual.
Assume that $p$ does not divide
$|H|$ so that $\mathcal{G} := \Gal(L_{\infty}/K)$ decomposes into
a direct product $H \times \Gamma$ with $\Gamma \simeq \Z_p$.
Let us denote the group of principal units in a local field $F$
by $U^1(F)$ and consider the inverse limit 
\[
	U^1(L_{\infty}) := \varprojlim_n U^1(L_n),
\]
where the transition maps are given by the norm maps. 
Moreover, we let $X_{L_{\infty}}$ be the Galois group over $L_{\infty}$
of the maximal abelian pro-$p$-extension of $L_{\infty}$.
Then both $U^1(L_{\infty})$ and $X_{L_{\infty}}$ are finitely generated
$\Lambda(\mathcal{G})$-modules.
If $L$ contains a primitive $p$-th root of unity,
then by \cite[Theorems 11.2.3 and 11.2.4]{MR2392026} there are
(non-canonical) isomorphisms
of $\Lambda(\mathcal{G})$-modules
\begin{equation} \label{eqn:intro}
	X_{L_{\infty}} \simeq 
	U^1(L_{\infty}) \simeq 
	\Z_p(1) \oplus \Lambda(\mathcal{G})^{[K : \Q_p]}
\end{equation}
and without the $\Z_p(1)$-term otherwise. Similar statements in fact hold
for more arbitrary $\Z_p$-extensions of $L$. 
However, this does not remain true
if $\mathcal{G}$ contains an element of order $p$ (this follows from the
results recalled in \S \ref{subsec:Galois-coh} and in particular
from sequence \eqref{eqn:4term-sequence}, since $\Z_p$ then has
infinite projective dimension as a $\Lambda(\mathcal{G})$-module).
In this case, the structure of these $\Lambda(\mathcal{G})$-modules
has not yet been determined, although this is a very natural and 
important question of Iwasawa theory.

Now let $H$ be arbitrary, but assume for simplicity in the introduction 
that $\mathcal{G} = H \times \Gamma$ is a direct product.
Recall that a homomorphism 
of finitely generated $\Lambda(\Gamma)$-modules is a pseudo-isomorphism
if and only if for every height $1$ prime ideal $\mathfrak p$
of $\Lambda(\Gamma)$ it becomes an isomorphism after localization 
at $\mathfrak p$. We will show that \eqref{eqn:intro} remains true
after localization at such a prime ideal $\mathfrak p$.
This is of independent interest, but we point out that
our motivation originates from the equivariant Iwasawa
main conjecture for local fields formulated by the author in
\cite{local-mc}. The inverse limit of the principal units
along the unramified $\Z_p$-tower naturally appears as a cohomology group
of a certain perfect complex of $\Lambda(\mathcal{G})$-modules, which plays
a key role in the formulation of this conjecture. It has been shown
\cite[Corollary 6.7]{local-mc} that it suffices to prove the 
conjecture after localization at the height $1$ prime ideal $(p)$.
For this reason we are interested in the $\Lambda_{(p)}(\mathcal{G})$-module
structure of the localization of $U^1(L_{\infty})$ at $(p)$,
where for any height $1$ prime ideal
$\mathfrak p$ we denote the localization
of $\Lambda(\mathcal{G})$ at $\mathfrak p$ by
$\Lambda_{\mathfrak{p}}(\mathcal{G})$. 

Our method is not restricted to the local case. We also consider
finite Galois extensions $L/K$ of number fields and the cyclotomic
$\Z_p$-extension $L_{\infty}$ of $L$. Then $\mathcal{G} :=
\Gal(L_{\infty}/K)$ is again a one-dimensional $p$-adic Lie group.
Let $S$ be a finite set of places of $K$ containing all the archimedean
places and all places that ramify in $L_{\infty}/K$. We then determine
the $\Lambda_{\mathfrak{p}}(\mathcal{G})$-module structure of the inverse
limit of the ($p$-completion of the) $S$-units, localized at $\mathfrak p$.
We also consider the natural Iwasawa module $X_S$, the Galois group over
$L_{\infty}$ of the maximal abelian pro-$p$-extension unramified outside $S$.

Our method has two main ingredients: The homotopy theory of Iwasawa
modules developed by Jannsen \cite{MR1097615} and, as a new ingredient,
an Iwasawa-theoretic analogue of a theorem of Swan \cite{MR0138688}.
The latter states that for a finite group $G$ two projective $\Z_p[G]$-modules
$P$ and $P'$ are isomorphic if and only if
$\Q_p \otimes_{\Z_p} P$ and
$\Q_p \otimes_{\Z_p} P'$ are isomorphic as $\Q_p[G]$-modules.
Accordingly, we prove that two finitely generated projective
$\Lambda_{\mathfrak{p}}(\mathcal{G})$-modules are isomorphic if and only
if this is true after base change to $\mathcal{Q}(\mathcal{G})$,
the total ring of fractions of $\Lambda(\mathcal{G})$ and thus also of
$\Lambda_{\mathfrak{p}}(\mathcal{G})$. This then allows us to compute the
projective summands of our Iwasawa modules. 

If $\mathcal{G} = H \times \Gamma$ is a direct product, then our result
is an easy consequence of Swan's original theorem. This is because then
$\Lambda(\mathcal{G})$ is obtained from the group ring $\Z_p[H]$ by extension
of scalars. However, the case of a semi-direct product is much harder,
and in fact our result cannot be directly deduced from Swan's theorem
or even from a more general result due to Hattori \cite{MR0175950} (see Remark
\ref{rem:Hattori} for details).

One method of proving Swan's theorem is via the Cartan--Brauer triangle,
since the Cartan map is injective in this case by a theorem of Brauer. This
method may be found in \cite[\S 21]{MR632548} and we largely follow
this approach. In fact, we construct a `Cartan--Brauer square' in a
rather abstract situation and show that the injectivity of the Cartan map
always implies a result in the style of Swan's theorem. The case of localized
Iwasawa algebras is then implied by a theorem of
Ardakov and Wadsley \cite{MR2585122} on the Cartan map
of crossed product algebras. As a by-product we deduce the surjectivity
of certain connecting homomorphisms that appear in relative
$K$-theory of Iwasawa algebras.\\

This article is organized as follows.
In \S \ref{sec:Swan} we first construct the Cartan--Brauer square,
a generalization of the Cartan--Brauer triangle in the case of group rings.
We then propose an abstract version of Swan's theorem (Corollary
\ref{cor:abstract-Swan}).
Viewing the localized Iwasawa algebras as crossed products allows us to
deduce our Iwasawa-theoretic analogue of Swan's theorem from the
aforementioned result of Ardakov and Wadsley 
(see Corollary \ref{cor:Swan-Iwasawa}).
In \S \ref{sec:homotopy} we review the homotopy theory of Iwasawa modules
and prove several auxiliary results for later use.
In \S \ref{sec:local} we study the Iwasawa theory of local fields.
In particular, our analogue of Swan's theorem allows us to show that
\eqref{eqn:intro} remains true for arbitrary one-dimensional
$p$-adic Lie extensions of $K$
after localization at an arbitrary height
$1$ prime ideal. Finally, we consider cyclotomic $\Z_p$-extensions
of number fields in \S \ref{sec:global}, where we prove analogues of
\cite[Theorem 11.3.11]{MR2392026}
for arbitrary one-dimensional $p$-adic Lie extensions
containing the cyclotomic $\Z_p$-extension.

\subsection*{Acknowledgements}
The author acknowledges financial support provided by the 
Deutsche Forschungsgemeinschaft (DFG) 
within the Heisenberg programme (No.\, NI 1230/3-1).
I thank Henri Johnston for his valuable comments on
an earlier version of this article and the referee for his/her
careful reading of this manuscript.

\subsection*{Notation and conventions}
All rings are assumed to have an identity element and all modules are assumed
to be left modules unless otherwise stated.
If $K$ is a field, we denote its absolute Galois group by $G_K$.
If $R$ is a ring and $M$ is an $R$-module,
we let $\pd_R(M)$ be the projective dimension
of $M$ over $R$.

\section{A generalization of Swan's theorem} \label{sec:Swan}

\subsection{Grothendieck groups} 
For further details and background on Grothendieck groups and
algebraic $K$-theory 
used in this section, we refer the reader to
\cite{MR892316} and \cite{MR0245634}.
Let $\Lambda$ be a noetherian ring and 
$\Mod(\Lambda)$ be the category of all 
$\Lambda$-modules. We denote the full subcategories of
all finitely generated and finitely generated projective $\Lambda$-modules by
$\Mod^{fg}(\Lambda)$ and $\PMod(\Lambda)$, respectively.
We let $G_0(\Lambda)$ and $K_{0}(\Lambda)$ be the Grothendieck groups
of $\Mod^{fg}(\Lambda)$ and $\PMod(\Lambda)$, respectively 
(see \cite[\S 38]{MR892316}).
The natural inclusion functor $\PMod(\Lambda) \rightarrow
\Mod^{fg}(\Lambda)$ induces a homomorphism
\[
	c: K_0(\Lambda) \longrightarrow G_0(\Lambda)
\]
which is called the \emph{Cartan map} or
the \emph{Cartan homomorphism}. We recall the following
result (see \cite[Proposition 38.22]{MR892316}).

\begin{lemma} \label{lem:equality-in-K0}
	Let $P, P' \in \PMod(\Lambda)$. Then we have
	$[P] = [P']$ in $K_0(\Lambda)$ if and only if
	$P \oplus Q \simeq P' \oplus Q$ for some $Q \in \PMod(\Lambda)$.
\end{lemma}

We write $K_{1}(\Lambda)$ for the Whitehead group of $\Lambda$,
which is the abelianized infinite general linear group
(see \cite[\S 40]{MR892316}).
We denote the relative algebraic $K$-group corresponding to a ring
homomorphism $\Lambda \rightarrow \Lambda'$ by $K_0(\Lambda, \Lambda')$.
We recall that $K_{0}(\Lambda, \Lambda')$ is 
an abelian group with generators $[X,g,Y]$ where
$X$ and $Y$ are finitely generated projective $\Lambda$-modules
and $g:\Lambda' \otimes_{\Lambda} X \rightarrow \Lambda' \otimes_{\Lambda} Y$ 
is an isomorphism of $\Lambda'$-modules;
for a full description in terms of generators and relations, 
we refer the reader to \cite[p.\ 215]{MR0245634}.
Furthermore, there is a long exact sequence of relative $K$-theory
(see  \cite[Chapter 15]{MR0245634})
\begin{equation} \label{eqn:long-exact-seq}
K_{1}(\Lambda) \longrightarrow K_{1}(\Lambda') \xrightarrow{\partial_{\Lambda,\Lambda'}}
K_{0}(\Lambda,\Lambda') \longrightarrow K_{0}(\Lambda) 
\longrightarrow K_{0}(\Lambda').
\end{equation}

\subsection{The decomposition map}
Let $R$ be a discrete valuation ring with maximal ideal $\mathfrak m$
and uniformizer $\pi$. We denote the field of fractions of $R$ by $K$
and let $k := R / \mathfrak m$ be the residue field.
Let $A$ be a finite dimensional $K$-algebra and let $\Lambda$
be an $R$-order in $A$. We put $\Omega := k \otimes_R \Lambda$,
which is a finite dimensional $k$-algebra.
Note that $A$ and $\Omega$ are artinian (and thus noetherian) 
rings so that every
finitely generated module has a composition series and satisfies
the Jordan--H\"older theorem \cite[Theorem 3.11]{MR632548}.

We also observe that every finitely generated $A$-module $V$ contains a full
$\Lambda$-lattice. Indeed, if $v_1, \dots, v_m$ is a $K$-basis
of $V$, then $M := \sum_{i=1}^m \Lambda v_i$ is a 
$\Lambda$-submodule of $V$ such that $K \otimes_R M = V$.
As $R$ is a discrete valuation ring, every finitely generated
torsionfree $R$-module is in fact free, and so $M$ is a full 
$\Lambda$-lattice in $V$.
We put $\overline{M} := M / \mathfrak m M = k \otimes_R M$,
which is a finitely generated $\Omega$-module.

\begin{prop} \label{prop:decomposition-map}
	There is a unique homomorphism of abelian groups
	\[
		d: G_0(A) \longrightarrow G_0(\Omega)
	\]
	such that for each finitely generated $A$-module $V$ one has
	$d([V]) = [\overline{M}]$, where $M$ is any full $\Lambda$-lattice
	in $V$.
\end{prop}

\begin{definition}
	The homomorphism $d: G_0(A) \rightarrow G_0(\Omega)$ in Proposition
	\ref{prop:decomposition-map} is called the 
	\emph{decomposition map}.
\end{definition}

\begin{proof}[Proof of Proposition \ref{prop:decomposition-map}]
	The proof is similar to that of \cite[Proposition 16.17]{MR632548},
	where the case of group rings is considered. 
	We repeat the argument for
	convenience of the reader.
	
	Let $V$ be a finitely generated $A$-module and choose a full
	$\Lambda$-lattice $M$ in $V$. We first show that the class
	$[\overline{M}]$ in $G_0(\Omega)$ does not depend on the choice of $M$.
	For this let $N$ be a second full $\Lambda$-lattice in $V$.
	By \cite[Proposition 16.6]{MR632548} we have 
	$[\overline{M}] = [\overline{N}]$ in $G_0(\Omega)$ if and only if
	$\overline{M}$ and $\overline{N}$ have the same composition factors.
	As $M+N$ is also a full lattice in $V$, we may assume
	that $N$ is properly contained in $M$. Since $M$ is noetherian,
	we may in addition assume that $N$ is a maximal $\Lambda$-submodule
	of $M$. We claim that $\pi M \subseteq N$. Otherwise, the chain 
	of inclusions $N \subsetneq N + \pi M \subseteq M$ gives
	$N + \pi M = M$ by maximality of $N$. Then
	Nakayama's Lemma implies $N = M$, contrary to our assumption.
	Now consider the chain of inclusions
	\[
		\pi N \subseteq \pi M \subseteq N \subseteq M.
	\]
	We see that $\overline{M}$ and $\overline{N}$ share the composition
	factors of $N / \pi M$. Thus it suffices to show that
	$M/N$ and $\pi M / \pi N$ have the same composition factors;
	but this is clear as multiplication by $\pi$ induces an isomorphism
	$M/N \simeq \pi M / \pi N$.
	
	Now define $d$ by $d([V]) = [\overline{M}]$. We have to show
	that $d$ is additive on short exact sequences. Given a short
	exact sequence of finitely generated $A$-modules
	\[
		0 \longrightarrow V_1 \longrightarrow V_2 
		\stackrel{\phi}{\longrightarrow} V_3  
		\longrightarrow 0,
	\]
	choose a full $\Lambda$-lattice $M_2$ in $V_2$ and define
	$M_3 := \phi(M_2)$ and $M_1 := M_2 \cap V_1$. Then we have a short exact
	sequence of $\Lambda$-modules
	\begin{equation} \label{eqn:ses-lattices}
	0 \longrightarrow M_1 \longrightarrow M_2 
	\stackrel{\phi}{\longrightarrow} M_3  
	\longrightarrow 0,
	\end{equation}
	and it is not hard to see that $M_1$ and $M_3$ are full
	$\Lambda$-lattices in $V_1$ and $V_3$, respectively.
	As $M_3$ is a free $R$-module, tensoring sequence
	\eqref{eqn:ses-lattices} with $k$ preserves exactness
	so that we obtain a short exact sequence of $\Omega$-modules
	\[
	0 \longrightarrow \overline{M_1} \longrightarrow \overline{M_2}
	\longrightarrow \overline{M_3} \longrightarrow 0.
	\]
	Thus we get 
	\[
		d([V_2]) = [\overline{M_2}] =
		[\overline{M_1}] + [\overline{M_3}] = d([V_1]) + d([V_3])
	\]
	as desired.
\end{proof}

\subsection{The Cartan--Brauer square}
We denote the radical of a ring $S$ by $\rad(S)$. We put
$\tilde{\Lambda} := \Lambda/ \rad(\Lambda) = \Omega/ \rad(\Omega)$.
Then $\tilde{\Lambda}$ is a semisimple artinian ring,
and $\Lambda$ is \emph{semiperfect} if and only if every idempotent
in $\tilde{\Lambda}$ is the image of an idempotent in $\Lambda$.
Note that $\Omega$ is always semiperfect 
by \cite[Propositions 6.5 and 6.7]{MR632548}.

\begin{remark}
	The ring $\Lambda$ is semiperfect whenever $R$ is complete
	\cite[Propositions 6.5 and 6.7]{MR632548} or $A$ is a split
	semisimple $K$-algebra \cite[Exercise 16]{MR632548}.
\end{remark}

Let us consider the following commutative square
\begin{equation} \label{eqn:Cartan-diagram}
\xymatrix{
	K_0(\Lambda) \ar[r]^b \ar[d]^{e} & K_0(\Omega) \ar[d]^c \\
	G_0(A) \ar[r]^d & G_0(\Omega),
}
\end{equation}
where for $P \in \PMod(\Lambda)$ we have $b([P]) = [\overline P]$
and $e([P]) = [K \otimes_R P]$. We call \eqref{eqn:Cartan-diagram}
the \emph{Cartan--Brauer square}.

\begin{prop} \label{prop:b-injective}
	The homomorphism $b: K_0(\Lambda) \rightarrow K_0(\Omega)$
	is injective. If $\Lambda$ is semiperfect, then $b$ is an
	isomorphism.
\end{prop}

\begin{proof}
	Let $P,P' \in \PMod(\Lambda)$ and assume that $[\overline{P}]
	=[\overline{P'}]$ in $K_0(\Omega)$. 
	By Lemma \ref{lem:equality-in-K0} there exists
	an $S\in \PMod(\Omega)$ such that
	$\overline{P} \oplus S \simeq \overline{P'} \oplus S$.
	We may assume that $S$ is free and thus in particular that
	$S \simeq \overline Q$ for some $Q \in \PMod(\Lambda)$.
	We claim that $P \oplus Q \simeq P' \oplus Q$.
	Then clearly $[P] = [P']$ in $K_0(\Lambda)$ and thus $b$ is injective.
	For the claim we may assume that $R$ is complete by
	\cite[Proposition 30.17]{MR632548} in which case it follows from
	\cite[Proposition 6.17 (iv)]{MR632548}.
	
	Now suppose that $\Lambda$ is semiperfect and let
	$Q \in \PMod(\Omega)$. In order to show that $b$ is surjective,
	it suffices to find $P \in \PMod(\Lambda)$ such that
	$\overline{P} \simeq Q$. 
	Let us put $\tilde Q := Q / \rad(\Omega)Q \in \PMod(\tilde{\Lambda})$.
	Then there is a $P \in \PMod(\Lambda)$ such that
	$P / \rad(\Lambda) P \simeq \tilde Q$ by \cite[Theorem 6.23]{MR632548}.
	Then both $\overline{P}$ and $Q$ are finitely generated projective 
	$\Omega$-modules and projective covers of $\tilde Q$
	by \cite[Corollary 6.22]{MR632548}. This implies
	$\overline{P} \simeq Q$ as projective covers are unique
	up to isomorphism \cite[Proposition 6.20]{MR632548}.
\end{proof}

\begin{remark}
	If $G$ is a finite group such that the group ring $R[G]$ is semiperfect,
	diagram \eqref{eqn:Cartan-diagram} specializes to the
	\emph{Cartan--Brauer triangle} (see \cite[\S 18A]{MR632548})
	\[ \xymatrix{
		& K_0(k[G])	\ar[dl]_{e b^{-1}} \ar[dr]^c & \\
		G_0(K[G]) \ar[rr]^d & & G_0(k[G]).
	}\]
\end{remark}

The following result might be seen as an abstract version of 
Swan's theorem \cite[\S 6]{MR0138688} 
(see also \cite[Theorem 32.1]{MR632548}).

\begin{corollary} \label{cor:abstract-Swan}
	Let $P,P' \in \PMod(\Lambda)$ and suppose that the Cartan map
	$c$ is injective.
	Then $P \simeq P'$ as $\Lambda$-modules if and only if
	$K \otimes_R P \simeq K \otimes_R P'$ as $A$-modules.
\end{corollary}

\begin{proof}
	As the map $b$ is injective by Proposition \ref{prop:b-injective}
	and the Cartan map $c$ is injective by assumption, also the map
	$e$ in diagram \eqref{eqn:Cartan-diagram} has to be injective.
	Now assume that
	$K \otimes_R P \simeq K \otimes_R P'$ as $A$-modules.
	Then we have in particular that $e([P]) = e([P'])$ in 
	$G_0(A)$ and thus $[P] = [P']$ in $K_0(\Lambda)$.
	By Lemma \ref{lem:equality-in-K0} there is a finitely generated
	projective $\Lambda$-module $Q$ such that $P \oplus Q \simeq
	P' \oplus Q$. In order to deduce $P \simeq P'$
	we may assume that $R$ is complete by \cite[Proposition 30.17]{MR632548}.
	Now the result follows from \cite[Corollary 6.15]{MR632548}.
\end{proof}

\begin{corollary}[Swan] \label{cor:Swan}
	Let $G$ be a finite group and let $P, P' \in \PMod(R[G])$.
	Then $P \simeq P'$ as $R[G]$-modules if and only if
	$K \otimes_R P \simeq K \otimes_R P'$ as $K[G]$-modules.
\end{corollary}

\begin{proof}
	It suffices to show that the Cartan map is injective.
	If $k$ has positive characteristic, this follows from
	a theorem of Brauer
	(see \cite[Theorem 21.22]{MR632548} or
	\cite[Corollary 1 of Theorem 35]{MR0450380}).
	If $k$ has characteristic zero (or if the characteristic
	is positive and does not divide the cardinality of $G$),
	then $k[G]$ is a semisimple ring by Maschke's theorem
	\cite[Theorem 3.14]{MR632548}.
	Thus every finitely generated $k[G]$-module is indeed projective
	and the Cartan map becomes the identity morphism.
\end{proof}

In view of Corollary \ref{cor:abstract-Swan} it is an interesting
question of study in which cases the Cartan map is injective.
For this the following observation will be very useful.

\begin{lemma} \label{lem:rank-K-groups}
	Let $s$ be the number of non-isomorphic simple (left) $\Omega$-modules
	of an arbitrary (left) artinian ring $\Omega$.
	Then the abelian groups $K_0(\Omega)$ and $G_0(\Omega)$
	are free $\Z$-modules of rank $s$.
\end{lemma}

\begin{proof}
	Let $s'$ be the number of non-isomorphic indecomposable
	left ideals in $\Omega$.
	As $\Omega$ is an artinian ring, the groups 
	$G_0(\Omega)$ and $K_0(\Omega)$
	are free $\Z$-module of rank $s$ and $s'$ by 
	\cite[Propositions 16.6 and 16.7]{MR632548}, respectively.
	However, if $I$ is an indecomposable left ideal in $\Omega$,
	then $\tilde I := I / \rad(\Omega) I$ is a simple left module
	by \cite[Corollary 6.9]{MR632548}, and $I$ is the projective cover
	of $\tilde I$ by \cite[Corollary 6.22]{MR632548}.
	This induces a one-to-one correspondence between the
	indecomposable left ideals and the simple left modules
	(see \cite[\S 6B]{MR632548} and in particular 
	\cite[Proposition 6.17]{MR632548}). Thus we have
	$s = s'$ as desired.
\end{proof}

\subsection{Crossed products}
Let $G$ be a finite group and let $R$ be a ring. Recall from
\cite[1.5.8]{MR1811901} that a \emph{crossed product} of $R$ by $G$
is an associative ring $R \ast G$ which contains $R$ as a subring
and a set of units $U_G = \left\{u_g \mid g \in G \right\}$ of 
cardinality $|G|$ such that
\begin{enumerate}
	\item 
	$R \ast G$ is a free $R$-module with basis $U_G$;
	\item
	for all $g,h \in G$ one has $u_g R = R u_g$ and
	$u_g \cdot u_h R = u_{gh} R$.
\end{enumerate}

We need the following result which immediately follows
from Lemma \ref{lem:rank-K-groups} and a theorem of Ardakov and Wadsley
\cite[\S 1.1]{MR2585122} (where Brauer's theorem again
appears as a key step in the proof).

\begin{theorem} \label{thm:Ardakov-Wadsley}
	Let $G$ be a finite group and let $k$ be a field. Then for every
	crossed product of $k$ by $G$, the Cartan map
	\[
		c: K_0(k \ast G) \longrightarrow G_0(k \ast G)
	\]
	is injective with finite cokernel.
\end{theorem}

\subsection{Iwasawa algebras} \label{subsec:iwasawa-alg}
Let $p$ be a prime and
$\mathcal{G}$ be a profinite group.
The complete group algebra of $\mathcal{G}$ over $\Z_p$ is 
\[
\Lambda(\mathcal{G}) := \Z_{p}\llbracket \mathcal{G} \rrbracket = \varprojlim \Z_{p}[\mathcal{G}/\mathcal{N}],
\]
where the inverse limit is taken over all open normal subgroups $\mathcal{N}$ of $\mathcal{G}$. Then $\Lambda(\mathcal{G})$ 
is a compact $\Z_p$-algebra and we denote
the kernel of the natural augmentation map
$\Lambda(\mathcal{G}) \twoheadrightarrow \Z_p$ by $\Delta(\mathcal{G})$.
If $M$ is a (left) $\Lambda(\mathcal{G})$-module we let
$M_{\mathcal{G}} := M / \Delta(\mathcal{G}) M$ be the module of coinvariants
of $M$. This is the maximal quotient module of $M$ with trivial
$\mathcal{G}$-action. 
Similarly, we denote the maximal submodule of $M$ upon which
$\mathcal{G}$ acts trivially by $M^{\mathcal{G}}$.

Now suppose that $\mathcal{G}$
contains a finite normal subgroup $H$ 
such that $\mathcal{G}/H \simeq \Z_p$.
Then $\mathcal{G}$ may be written as a semi-direct product 
$\mathcal{G} = H \rtimes \Gamma$
where $\Gamma \simeq \Z_{p}$. 
In other words, $\mathcal{G}$ is  a one-dimensional $p$-adic Lie group.

If $F$ is a finite field extension of $\Q_{p}$  with ring of integers 
$\mathcal{O}=\mathcal{O}_{F}$, we put 
$\Lambda^{\mathcal{O}}(\mathcal{G}) := \mathcal{O} \otimes_{\Z_{p}} \Lambda(\mathcal{G}) = \mathcal{O}\llbracket \mathcal{G} \rrbracket$.
We fix a topological generator $\gamma$ of $\Gamma$.
Since any homomorphism $\Gamma \rightarrow \Aut(H)$ must have open kernel, 
we may choose a natural number $n$ such that $\gamma^{p^n}$ is central in $\mathcal{G}$; 
we fix such an $n$.
As $\Gamma_{0} := \Gamma^{p^n} \simeq \Z_{p}$, there is a ring isomorphism
$\Lambda^{\mathcal{O}}(\Gamma_{0}) \simeq \mathcal{O}\llbracket T \rrbracket$ 
induced by $\gamma^{p^n} \mapsto 1+T$
where $\mathcal{O}\llbracket T \rrbracket$ denotes the power series ring in one variable over $\mathcal{O}$.
If we view $\Lambda^{\mathcal{O}}(\mathcal{G})$ as a
$\Lambda^{\mathcal{O}}(\Gamma_{0})$-module 
(or indeed as a left $\Lambda^{\mathcal{O}}(\Gamma_{0})[H]$-module), 
there is a decomposition
\begin{equation} \label{eq:Lambda-decomp}
\Lambda^{\mathcal{O}}(\mathcal{G}) = 
\bigoplus_{i=0}^{p^n-1} \Lambda^{\mathcal{O}}(\Gamma_{0})[H] \gamma^{i}.
\end{equation}
Hence $\Lambda^{\mathcal{O}}(\mathcal{G})$ is finitely generated as an $\Lambda^{\mathcal{O}}(\Gamma_{0})$-module
and is an $\Lambda^{\mathcal{O}}(\Gamma_{0})$-order 
in the separable $\mathcal{Q}^{F} (\Gamma_0) 
:= Quot(\Lambda^{\mathcal{O}}(\Gamma_{0}))$-algebra
$\mathcal{Q}^{F} (\mathcal{G})$, the total ring of fractions of 
$\Lambda^{\mathcal{O}}(\mathcal{G})$, obtained
from $\Lambda^{\mathcal{O}}(\mathcal{G})$ by adjoining inverses of all central regular elements.
It follows from \eqref{eq:Lambda-decomp} that
$\Lambda^{\mathcal{O}}(\mathcal{G})$
is a crossed product 
of $\Lambda^{\mathcal{O}}(\Gamma_{0})$
by $\mathcal{G} / \Gamma_0$ 
(see also \cite[\S 2.3]{MR2290583}):
\[
	\Lambda^{\mathcal{O}}(\mathcal{G}) \simeq
	\Lambda^{\mathcal{O}}(\Gamma_{0}) \ast (\mathcal{G} / \Gamma_0).
\]

The commutative ring $\Lambda^{\mathcal{O}}(\Gamma_0)$ 
is a regular local ring of dimension $2$. 
If $\mathfrak p$ is a prime ideal in $\Lambda^{\mathcal{O}}(\Gamma_0)$
of height $1$, we denote the localization of
$\Lambda^{\mathcal{O}}(\Gamma_0)$ at
$\mathfrak p$ by $\Lambda_{\mathfrak p}^{\mathcal{O}}(\Gamma_0)$. This is a
discrete valuation ring and we denote its residue field by
$\Omega_{\mathfrak p}^{\mathcal{O}}(\Gamma_0)$. We also put
\begin{eqnarray*}
	\Lambda_{\mathfrak p}^{\mathcal{O}}(\mathcal G)  := & 
	\Lambda_{\mathfrak p}^{\mathcal{O}}(\Gamma_0)
	\otimes_{\Lambda^{\mathcal{O}}(\Gamma_0)}
	\Lambda^{\mathcal{O}}(\mathcal{G}) & \simeq 
	\Lambda_{\mathfrak p}^{\mathcal{O}}(\Gamma_0) \ast 
		(\mathcal{G} / \Gamma_0)\\
	\Omega_{\mathfrak p}^{\mathcal{O}}(\mathcal G)  := & 
	\Omega_{\mathfrak p}^{\mathcal{O}}(\Gamma_0)
	\otimes_{\Lambda^{\mathcal{O}}(\Gamma_0)}
	\Lambda^{\mathcal{O}}(\mathcal{G}) & \simeq 
	\Omega_{\mathfrak p}^{\mathcal{O}}(\Gamma_0) \ast 
		(\mathcal{G} / \Gamma_0).
\end{eqnarray*}
We therefore have the following special case of Theorem \ref{thm:Ardakov-Wadsley}.

\begin{prop} \label{prop:cartan-Iwasawa}
	Let $\mathfrak p$ be a prime ideal in 
	$\Lambda^{\mathcal{O}}(\Gamma_0)$ of height $1$.
	Then the Cartan map
	\[
		c: K_0(\Omega_{\mathfrak p}^{\mathcal{O}}(\mathcal G))
		\longrightarrow G_0(\Omega_{\mathfrak p}^{\mathcal{O}}(\mathcal G))
	\]
	is injective with finite cokernel.
\end{prop}

The analogue of Swan's theorem for Iwasawa algebras is now easily
established:

\begin{corollary} \label{cor:Swan-Iwasawa}
	Let $\mathfrak p$ be a prime ideal in 
	$\Lambda^{\mathcal{O}}(\Gamma_0)$ of height $1$
	and let $P,P' \in \PMod(\Lambda_{\mathfrak p}^{\mathcal{O}}(\mathcal G))$.
	Then $P \simeq P'$ as 
	$\Lambda_{\mathfrak p}^{\mathcal{O}}(\mathcal G)$-modules if and only if
	$\mathcal{Q}^{F} (\mathcal{G}) \otimes_{\Lambda_{\mathfrak p}^{\mathcal{O}}(\mathcal G)} P \simeq
	\mathcal{Q}^{F} (\mathcal{G}) \otimes_{\Lambda_{\mathfrak p}^{\mathcal{O}}(\mathcal G)} P'$ as
	$\mathcal{Q}^{F} (\mathcal{G})$-modules.
\end{corollary}

\begin{proof}
	This immediately follows from Corollary \ref{cor:abstract-Swan} and
	Proposition \ref{prop:cartan-Iwasawa}.
\end{proof}

\begin{remark} \label{rem:Hattori}
	If $\mathcal{G} = H \times \Gamma$ is a direct product, then
	Corollary \ref{cor:Swan-Iwasawa} is a direct consequence of
	Swan's original theorem (Corollary \ref{cor:Swan}).
	We now explain why even Hattori's more general approach
	\cite{MR0175950} (see \cite[Theorem 32.5]{MR632548}) 
	to Swan's theorem does not imply Corollary
	\ref{cor:Swan-Iwasawa} if $\mathcal{G} = H \rtimes \Gamma$
	is only a semi-direct product.
	
	Assume for simplicity that $\mathcal{G}$ is a pro-$p$-group
	and that $\mathcal{O} = \Z_p$. 
	If $\mathcal{G} = H \rtimes \Gamma$ is not a direct product,
	then any choice of $\Gamma_0$ will be a proper subgroup of $\Gamma$.
	Let $\Delta(H)$ be the 
	(left) ideal of $\Omega_{(p)}(\mathcal{G})$ generated by
	the elements $h-1$, $h \in H$. Then $\Delta(H)$ is nilpotent
	and thus contained in the radical $\mathfrak r :=
	\rad(\Omega_{(p)}(\mathcal{G}))$ by \cite[Proposition 5.15]{MR632548}.
	However, we have that
	\[
	\Omega_{(p)}(\mathcal{G}) / \Delta(H) \simeq
	\Omega_{(p)}(\Gamma_0) \ast \Gamma/ \Gamma_0 \simeq
	\Omega_{(p)}(\Gamma)
	\]
	is an inseparable field extension of $\Omega_{(p)}(\Gamma_0)$. 
	Hence $\mathfrak r = \Delta(H)$ and $\Omega_{(p)}(\mathcal{G}) /
	\mathfrak r$ is not a separable $\Omega_{(p)}(\Gamma_0)$-algebra
	as it would be required for Hattori's theorem.
\end{remark}

\begin{corollary} \label{cor:connecting-hom}
	Let $\mathfrak p$ be a prime ideal in 
	$\Lambda^{\mathcal{O}}(\Gamma_0)$ of height $1$. Then the connecting
	homomorphism
	\[
		\partial_{\Lambda_{\mathfrak p}^{\mathcal{O}}(\mathcal G),
			\mathcal{Q}^F(\mathcal{G})}:
		K_1(\mathcal{Q}^F(\mathcal{G})) \longrightarrow 
		K_0(\Lambda_{\mathfrak p}^{\mathcal{O}}(\mathcal G),
		\mathcal{Q}^F(\mathcal{G}))
	\]
	is surjective.
\end{corollary}

\begin{proof}
	This follows from the long exact sequence 
	\eqref{eqn:long-exact-seq} and Corollary \ref{cor:Swan-Iwasawa}.
\end{proof}

\begin{corollary}
	Let $\mathfrak p$ be a prime ideal in 
	$\Lambda^{\mathcal{O}}(\Gamma_0)$ of height $1$. Then the connecting
	homomorphism
	\[
	\partial_{\Lambda^{\mathcal{O}}(\mathcal G), \Lambda_{\mathfrak p}^{\mathcal{O}}(\mathcal G)}:
	K_1(\Lambda_{\mathfrak p}^{\mathcal{O}}(\mathcal G)) \longrightarrow 
	K_0(\Lambda^{\mathcal{O}}(\mathcal G), \Lambda_{\mathfrak p}^{\mathcal{O}}(\mathcal G))
	\]
	is surjective. Moreover, we have a short exact sequence of abelian groups
	\[
	0 \longrightarrow K_0(\Lambda^{\mathcal{O}}(\mathcal G), \Lambda_{\mathfrak p}^{\mathcal{O}}(\mathcal G)) \longrightarrow  
	K_0(\Lambda^{\mathcal{O}}(\mathcal G),
	\mathcal{Q}^F(\mathcal{G})) \longrightarrow  
	K_0(\Lambda_{\mathfrak p}^{\mathcal{O}}(\mathcal G),
	\mathcal{Q}^F(\mathcal{G})) \longrightarrow 0.
	\]
\end{corollary}

\begin{proof}
	Consider the long exact sequences
	\eqref{eqn:long-exact-seq} for the three occurring pairs.
	The connecting homomorphism 
	\[
	\partial_{\Lambda^{\mathcal{O}}(\mathcal G), \mathcal{Q}^F(\mathcal G)}:
	K_1(\mathcal{Q}^F(\mathcal G)) \longrightarrow 
	K_0(\Lambda^{\mathcal{O}}(\mathcal G), \mathcal{Q}^F(\mathcal G))
	\]
	is surjective by \cite[Corollary 3.8]{MR3034286}.
	The result follows from Corollary \ref{cor:connecting-hom}
	by an easy diagram chase.
\end{proof}

\section{Homotopy theory} \label{sec:homotopy}

\subsection{Homotopy of modules}
We briefly recall basic material of homotopy theory of modules.
The reader may also consult Jannsen \cite[\S 1]{MR1097615} and
\cite[Chapter V, \S 4]{MR2392026}.

Let $\Lambda$ be a ring.
If a homomorphism $f: M \rightarrow N$ of $\Lambda$-modules 
factors through a projective $\Lambda$-module,
then we say that $f$ is \emph{homotopic to zero} and we write $f \sim 0$.
Two homomorphisms $f,g: M \rightarrow N$ are \emph{homotopic} ($f \sim g$)
if $f-g$ is homotopic to zero. We let $\Ho(\Lambda)$ be the
homotopy category of $\Lambda$-modules. This category has the same objects
as $\Mod(\Lambda)$, but the homomorphism groups are given by
$\Hom_{\Lambda}(M,N) / \left\{f \sim 0 \right\}$.
A homomorphism $f: M \rightarrow N$ of $\Lambda$-modules is a
\emph{homotopy equivalence} if it is an isomorphism in $\Ho(\Lambda)$.
In this case, we say that $M$ and $N$ are \emph{homotopy equivalent}
and write $M \sim N$.

For any (left) $\Lambda$-module $M$ and integer $i \geq 0$ 
we define (right) $\Lambda$-modules $M^+ := \Hom_{\Lambda}(M, \Lambda)$ and
$E^i(M) := \Ext_{\Lambda}^i(M, \Lambda)$. In particular, we have
$M^+ = E^0(M)$.
We denote the full subcategory of $\Ho(\Lambda)$ whose objects are
finitely presented $\Lambda$-modules by $\Ho^{fp}(\Lambda)$.
The \emph{transpose} is a contravariant functor
\[
	D: \Ho^{fp}(\Lambda) \longrightarrow \Ho^{fp}(\Lambda)
\]
that on objects is given as follows. 
Let $M$ be a finitely presented $\Lambda$-module and choose a presentation
\[
	P_1 \longrightarrow P_0 \longrightarrow M \longrightarrow 0
\]
by finitely generated projective $\Lambda$-modules. Then $DM$ is defined
by the exact sequence
\[
	0 \longrightarrow M^+ \longrightarrow P_0^+ \longrightarrow P_1^+ \longrightarrow
	DM \longrightarrow 0.
\]
The transpose is a contravariant autoduality of $\Ho^{fp}(\Lambda)$, 
i.e.~$D \circ D \simeq \id$, by \cite[Proposition 5.4.9]{MR2392026}.
Moreover, for every finitely presented $\Lambda$-module $M$ 
there is an exact sequence of $\Lambda$-modules
\begin{equation} \label{eqn:transpose-sequence}
	0 \longrightarrow E^1(DM) \longrightarrow M \xrightarrow{\phi_M} M^{++}
	\longrightarrow E^2(DM) \longrightarrow 0,
\end{equation}
where $\phi_M$ is the canonical map of $M$ to its bidual.

\subsection{Homotopy of Iwasawa modules}
Let $\mathcal{G}$ be a one-dimensional $p$-adic Lie group.
As in subsection \S \ref{subsec:iwasawa-alg} we choose a central
subgroup $\Gamma_0 \simeq \Z_p$ in $\mathcal{G}$ and view
$\Lambda(\mathcal{G})$ as a $\Lambda(\Gamma_0)$-order in 
$\mathcal{Q}(\mathcal{G})$. We denote the set of prime
ideals in $\Lambda(\Gamma_0)$ of height $1$ by $\boldsymbol P_0$.
We let $^{\sharp}: \mathcal Q(\mathcal{G}) \rightarrow
\mathcal Q(\mathcal{G})$ be the anti-involution that maps each group
element $g \in \mathcal{G}$ to its inverse.
For every $\mathfrak p \in \boldsymbol{P}_0$ we have
$\mathfrak p^{\sharp} := \left\{x^{\sharp} \mid x \in \mathfrak p \right\}
\in \boldsymbol{P}_0$ and in particular an equality
$(\Lambda_{\mathfrak p}(\mathcal{G}))^{\sharp} = 
\Lambda_{\mathfrak p^{\sharp}}(\mathcal{G})$.

The functors $D$ and $E^i$ interchange left and right $\Lambda$-action.
If $\Lambda = \Lambda(\mathcal{G})$ is the Iwasawa algebra,
then we have a natural equivalence between left and
right modules, induced by the anti-involution $^{\sharp}$.
We then endow $DM$ and $E^i(M)$ with this left module structure.
Namely, for $\lambda \in \Lambda(\mathcal{G})$ and $x \in DM$ or
$x \in E^i(M)$ we let $\lambda \cdot x  := x \cdot \lambda^{\sharp}$.
Similarly, if $\Lambda = \Lambda_{\mathfrak p}(\mathcal{G})$ for some
$\mathfrak p \in \boldsymbol{P}_0$, then for every finitely presented
left $\Lambda_{\mathfrak p}(\mathcal{G})$-module $M$, the transpose 
$DM$ and $E^i(M)$ are natural left 
$\Lambda_{\mathfrak p^{\sharp}}(\mathcal{G})$-modules.

The functors $D$ and $E^i$ then commute with localization in the sense
that for every prime ideal $\mathfrak p$ of $\Lambda(\Gamma_0)$ we have
$DM_{\mathfrak p} = (DM)_{\mathfrak p^{\sharp}}$ and
$E^i(M_{\mathfrak p}) = E^i(M)_{\mathfrak p^{\sharp}}$;
here and in the following the notation $DM_{\mathfrak p}$ always means
the transpose of $M_{\mathfrak p}$ and \emph{not} the localization
of $DM$ at $\mathfrak p$.
In particular, for every finitely generated $\Lambda(\mathcal{G})$-module
$M$ and every $\mathfrak p \in \boldsymbol{P}_0$
we have $E^2(M_{\mathfrak p}) = E^2(M)_{\mathfrak p^{\sharp}} = 0$
by \cite[Proposition 5.5.3]{MR2392026}. In fact,
we have the following result which will often be used without
reference.

\begin{lemma} \label{lem:E2-vanishes}
	Let $\mathcal{G}$ be a one-dimensional $p$-adic Lie group and let
	$\mathfrak p \in \boldsymbol{P}_0$.
	Then $E^2(M)$ vanishes for every finitely generated 
	$\Lambda_{\mathfrak p}(\mathcal{G})$-module $M$. In particular,
	there is an exact sequence
	\[
		0 \longrightarrow E^1(DM) \longrightarrow M \longrightarrow M^{++}
		\longrightarrow 0.
	\]
\end{lemma}

\begin{proof}
	The map $M / E^1(DM) \rightarrow M^{++}$
	induced by $\phi_M$ is an injective pseudo-isomorphism
	by (the proof of) \cite[Proposition 5.1.8]{MR2392026}.
	Then sequence \eqref{eqn:transpose-sequence} implies that
	$E^2(DM)$ is pseudo-null as a 
	$\Lambda_{\mathfrak p}(\Gamma_0)$-module and thus vanishes, since
	$\Lambda_{\mathfrak p}(\Gamma_0)$ is a discrete valuation ring.
	Applying this argument to $DM$, we obtain
	$E^2(M) \simeq E^2(DDM) = 0$.
\end{proof}

\begin{lemma} \label{lem:M+projective}
	Let $\mathcal{G}$ be a one-dimensional $p$-adic Lie group
	and let $\Lambda$ be either the Iwasawa algebra $\Lambda(\mathcal{G})$
	or $\Lambda_{\mathfrak p}(\mathcal{G})$ for some prime ideal 
	$\mathfrak p \in \boldsymbol{P}_0$. Let $M$ be a finitely
	generated $\Lambda$-module such that $M^{++}$ has
	finite projective dimension. Then
	the $\Lambda^{\sharp}$-module $M^+$ and 
	the $\Lambda$-module $M^{++}$ are indeed projective.
\end{lemma}

\begin{proof}
	We assume that $\Lambda = \Lambda(\mathcal{G})$; 
	the other case can be treated similarly.
	We put $d := \pd_{\Lambda(\mathcal{G})}(M^{++})$ and
	choose a projective resolution
	\[
		0 \longrightarrow P_d \longrightarrow \cdots 
		\longrightarrow P_1 \longrightarrow P_0 \longrightarrow
		M^{++} \longrightarrow 0.
	\]
	As $M^+$ and $M^{++}$ are reflexive and thus free as a 
	$\Lambda(\Gamma_0)$-modules
	by \cite[Propositions 5.1.9 and 5.4.17]{MR2392026}, this induces an
	exact sequence of $\Lambda(\mathcal{G})$-modules
	\[
		0 \longrightarrow M^+ \longrightarrow P_0^+ \longrightarrow P_1^+ \longrightarrow
		\cdots \longrightarrow P_d^+ \longrightarrow 0.
	\]
	As each $P_i^+$, $0 \leq i \leq d$ is a projective
	$\Lambda(\mathcal{G})^{\sharp}$-module, 
	so is $M^+$. The result follows.
\end{proof}

The next result shows that Lemma \ref{lem:M+projective} is only interesting
if $\Lambda = \Lambda(\mathcal{G})$ or $\Lambda = \Lambda_{(p)}(\mathcal{G})$.

\begin{lemma} \label{lem:reflexive-projective}
	Let $\mathfrak{p} \in \boldsymbol{P}_0$ be a prime ideal
	and assume that $\mathfrak p \not= (p)$. Let $M$ be a finitely generated
	$\Lambda_{\mathfrak p}(\mathcal{G})$-module.
	Then $M$ is a projective $\Lambda_{\mathfrak p}(\mathcal{G})$-module
	if and only if $M$ is (torsion-)free as a 
	$\Lambda_{\mathfrak p}(\Gamma_0)$-module.
	In particular, every reflexive $\Lambda_{\mathfrak p}(\mathcal{G})$-module
	is projective.
\end{lemma}

\begin{proof}
	We first recall that every torsionfree 
	(and in particular every projective)
	$\Lambda_{\mathfrak p}(\Gamma_0)$-module is in fact free,
	since $\Lambda_{\mathfrak p}(\Gamma_0)$ is a discrete valuation ring.
	Now suppose that $M$ is a projective 
	$\Lambda_{\mathfrak p}(\mathcal{G})$-module. As
	$\Lambda_{\mathfrak p}(\mathcal{G})$ is free as a
	$\Lambda_{\mathfrak p}(\Gamma_0)$-module, the module $M$
	is a submodule of a free $\Lambda_{\mathfrak p}(\Gamma_0)$-module
	and thus free.
	For the converse we put $G := \mathcal{G}/\Gamma_0$. 
	Then $G$ is a finite group
	and $|G|$ is invertible in $\Lambda_{\mathfrak p}(\Gamma_0)$
	since $\mathfrak p \not= (p)$. Thus
	$\Hom_{\Lambda_{\mathfrak p}(\Gamma_0)}(M,N)$ is a 
	$\Q_p[G]$-module for any two 
	$\Lambda_{\mathfrak p}(\mathcal{G})$-modules $M$ and $N$.
	Since taking $G$-invariants is an exact functor on $\Q_p[G]$-modules,
	the equality
	\[
		\Hom_{\Lambda_{\mathfrak p}(\mathcal G)}(M,N) = 
		\Hom_{\Lambda_{\mathfrak p}(\Gamma_0)}(M,N)^G
	\]
	implies isomorphisms
	\[
		\Ext^i_{\Lambda_{\mathfrak p}(\mathcal G)}(M,N) \simeq
		\Ext^i_{\Lambda_{\mathfrak p}(\Gamma_0)}(M,N)^G
	\]
	for all $i \geq 0$. This gives the converse implication.
\end{proof}

\begin{remark}
	Suppose that $\mathcal{G} \simeq H \times \Gamma$ and that $p$
	does not divide the cardinality of $H$. Then we can take 
	$\Gamma_0 = \Gamma$ and Lemma \ref{lem:reflexive-projective}
	remains true for $\mathfrak p = (p)$ and the Iwasawa algebra
	$\Lambda(\mathcal{G})$ by \cite[Lemma 5.4.16]{MR2392026}.
\end{remark}

\begin{corollary} \label{cor:pd1}
	Let $\mathfrak{p} \in \boldsymbol{P}_0$ be a prime ideal
	and assume that $\mathfrak p \not= (p)$. Then every finitely generated
	$\Lambda_{\mathfrak p}(\mathcal{G})$-module has projective dimension
	at most $1$.
\end{corollary}

\begin{corollary} \label{cor:M-splits}
	Let $\mathfrak{p} \in \boldsymbol{P}_0$ be a prime ideal
	and assume that $\mathfrak p \not= (p)$. Let $M$ be a finitely generated
	$\Lambda_{\mathfrak p}(\mathcal{G})$-module. Then there is an
	isomorphism 
	\[
		M \simeq E^1(DM) \oplus M^{++}.
	\]
\end{corollary}

\begin{proof}
	This follows from Lemma \ref{lem:E2-vanishes} and
	Lemma \ref{lem:reflexive-projective}.
\end{proof}

\begin{corollary} \label{cor:augideal-free}
	For every $\mathfrak p \in \boldsymbol{P}_0$ the 
	$\Lambda_{\mathfrak p}(\mathcal{G})$-module
		$\Delta(\mathcal{G})_{\mathfrak p}$ is free of rank $1$.
\end{corollary}

\begin{proof}
	We identify $\Lambda(\Gamma_0)$ with the power series ring
	$\Z_p \llbracket T \rrbracket$ as usual. If $\mathfrak p \not= (T)$
	then $(\Z_p)_{\mathfrak p}$ vanishes so that the exact sequence
	\[
	0 \longrightarrow \Delta(\mathcal{G}) \longrightarrow
	\Lambda(\mathcal{G}) \longrightarrow \Z_p \longrightarrow 0
	\]
	induces an isomorphism $\Delta(\mathcal{G})_{\mathfrak p} \simeq
	\Lambda_{\mathfrak p}(\mathcal{G})$. If $\mathfrak{p} = (T)$
	or more generally if $\mathfrak p \not= (p)$ then 
	$\Delta(\mathcal{G})_{\mathfrak p}$ is a projective
	$\Lambda_{\mathfrak p}(\mathcal{G})$-module by Lemma
	\ref{lem:reflexive-projective}. 
	Then Corollary \ref{cor:Swan-Iwasawa} implies
	that it is free of rank $1$ (in fact, an isomorphism
	$\Lambda_{\mathfrak p}(\mathcal{G}) \simeq 
	\Delta(\mathcal{G})_{\mathfrak p}$ is explicitly given by
	$1 \mapsto (1-\gamma) e_H + (1-e_H)$, where $e_H := |H|^{-1}
	\sum_{h \in H} h$).
\end{proof}

\section{Iwasawa theory of local fields} \label{sec:local}

\subsection{Galois cohomology} \label{subsec:Galois-coh}
If $F$ is a field and $M$ is a topological
$G_F$-module, we write $R\Gamma(F,M)$ for the complex of continuous 
cochains of $G_F$ with coefficients in $M$
and $H^i(F,M)$ for its cohomology in degree $i$. 
We likewise write $H_i(F,M)$ for the $i$-th homology group of $G_F$
with coefficients in $M$.
If $F$ is an algebraic extension
of $\Q_p$ or $\Q$ and $M$ is a discrete or compact $G_F$-module, 
then for $r \in \Z$
we denote the $r$-th Tate twist of $M$ by $M(r)$. For an abelian group
$A$ we write $\widehat A$ for its $p$-completion, that is
$\widehat A = \varprojlim_n A / p^n A$.

Let $L/K$ be a finite Galois extension of $p$-adic fields with 
Galois group $G$.
Let $L_{\infty}$ be an arbitrary $\Z_p$-extension of $L$
with Galois group $\Gamma_L$ and for each $n \in \N$ 
let $L_n$ be its $n$-th layer. We assume that
$L_{\infty} / K$ is again a Galois extension with Galois group
$\mathcal{G} := \Gal(L_{\infty} / K)$.
We let $X_{L_{\infty}}$ denote the Galois group
over $L_{\infty}$ of the maximal abelian pro-$p$-extension
of $L_{\infty}$. We put
\[
Y_{L_{\infty}} := \Delta(G_K)_{G_{L_{\infty}}} 
= \Z_p \widehat \otimes_{\Lambda(G_{L_{\infty}})} \Delta(G_K)
\]
and observe that $\pd_{\Lambda(\mathcal{G})}(Y_{L_{\infty}}) \leq 1$
by \cite[Theorem 7.4.2]{MR2392026}. As $H_1(L_{\infty}, \Z_p)$
canonically identifies with $X_{L_{\infty}}$, taking 
$G_{L_{\infty}}$-coinvariants of the obvious short exact sequence
\[
0 \longrightarrow \Delta(G_K) \longrightarrow \Lambda(G_K) \longrightarrow
\Z_p \longrightarrow 0
\]
yields an exact sequence
\begin{equation} \label{eqn:4term-sequence}
0 \longrightarrow X_{L_{\infty}} \longrightarrow Y_{L_{\infty}}
\longrightarrow \Lambda(\mathcal{G}) \longrightarrow \Z_p \longrightarrow 0
\end{equation}
of $\Lambda(\mathcal{G})$-modules (this should be compared to the 
sequence constructed by Ritter and Weiss \cite[\S 1]{MR1935024};
see also \cite[Proposition 5.6.7]{MR2392026}).
This sequence will be crucial in the following.

\begin{remark}
	The middle arrow in \eqref{eqn:4term-sequence} defines a (perfect)
	complex of $\Lambda(\mathcal{G})$-modules
	\[
		\cdots \longrightarrow 0 \longrightarrow
		Y_{L_{\infty}} \longrightarrow \Lambda(\mathcal{G})
		\longrightarrow 0 \longrightarrow \cdots
	\] 
	If we place $Y_{L_{\infty}}$
	in degree $1$, then this complex and
	$R\Hom(R\Gamma(L_{\infty}, \Q_p / \Z_p), \Q_p/ \Z_p)[-2]$ 
	become isomorphic in 
	the derived category of $\Lambda(\mathcal{G})$-modules 
	by \cite[Proposition 4.1]{local-mc}.
	If $L_{\infty}$ is the unramified $\Z_p$-extension of $L$,
	then this complex plays a key role in the equivariant
	Iwasawa main conjecture for local fields
	as formulated by the author \cite[Conjecture 5.1]{local-mc}.
	In order to verify this conjecture, 
	one may localize at the height $1$ prime ideal $(p)$
	by \cite[Corollary 6.7]{local-mc}. This has motivated our interest
	in the $\Lambda_{(p)}(\mathcal{G})$-module structure of
	$(X_{L_{\infty}})_{(p)}$.
\end{remark}

For any $p$-adic field $F$, we denote 
the group of principal units in $F$
by $U^1(F)$. 
We put $U^1(L_{\infty}) := \varprojlim_n U^1(L_n)$ where the 
transition maps are given by the norm maps. We note that
$\varprojlim_n \widehat{L_n^{\times}} \simeq X_{L_{\infty}}$
by local class field theory.
For each $n \geq 0$ the valuation map
$L_n^{\times} \rightarrow \Z$ induces an exact sequence
\[
	0 \longrightarrow U^1(L_n) \longrightarrow \widehat{L_n^{\times}} \longrightarrow
	\Z_p \longrightarrow 0.
\]
Taking inverse limits over all $n$ induces an exact sequence of
$\Lambda(\mathcal{G})$-modules
\begin{equation} \label{eqn:ramified-ses}
	0 \longrightarrow U^1(L_{\infty}) \longrightarrow X_{L_{\infty}}
	\longrightarrow \Z_p \longrightarrow 0
\end{equation}
if $L_{\infty}/L$ is ramified and an isomorphism
$U^1(L_{\infty}) \simeq X_{L_{\infty}}$ otherwise (also see the proof of
\cite[Theorem 11.2.4]{MR2392026}). 

\subsection{Local Iwasawa modules}
In this subsection we prove analogues of 
\cite[Theorems 11.2.3 and 11.2.4]{MR2392026}
for arbitrary one-dimensional $p$-adic Lie extensions.
As in subsection \S \ref{subsec:iwasawa-alg} we choose a central
subgroup $\Gamma_0 \simeq \Z_p$ in $\mathcal{G}$ and view
$\Lambda(\mathcal{G})$ as a $\Lambda(\Gamma_0)$-order in 
$\mathcal{Q}(\mathcal{G})$. 

\begin{lemma} \label{lem:first-prop}
	For every $\mathfrak p \in \boldsymbol{P}_0$ the following hold.
	\begin{enumerate}
		\item
		We have an isomorphism of $\Lambda_{\mathfrak p}(\mathcal{G})$-modules
		\[
			(Y_{L_{\infty}})_{\mathfrak p} \simeq 
			(X_{L_{\infty}})_{\mathfrak p} \oplus 
			\Lambda_{\mathfrak p}(\mathcal{G});
		\]
		\item
		we have $\pd_{\Lambda_{\mathfrak p}
			(\mathcal{G})}((X_{L_{\infty}})_{\mathfrak p}) 
		= \pd_{\Lambda_{\mathfrak p}(\mathcal{G})}((Y_{L_{\infty}})_{\mathfrak p}) \leq 1$.
	\end{enumerate}
\end{lemma}

\begin{proof}
	Sequence \eqref{eqn:4term-sequence} and Corollary
	\ref{cor:augideal-free} imply (i).
	For (ii) we compute
	\[
	\pd_{\Lambda_{\mathfrak p}(\mathcal{G})}((X_{L_{\infty}})_{\mathfrak p})
	= \pd_{\Lambda_{\mathfrak p}(\mathcal{G})}((Y_{L_{\infty}})_{\mathfrak p})
	\leq \pd_{\Lambda(\mathcal{G})}(Y_{L_{\infty}}) \leq 1.
	\]
\end{proof}

We denote the group of $p$-power
roots of unity in $L_{\infty}$ by $\mu_{p}(L_{\infty})$.
If $M$ is a $\Z_p$-module, we let $M^{\vee} := \Hom_{\Z_p}(M, \Q_p / \Z_p)$
be its Pontryagin dual. If $M$ is a $\mathcal{G}$-module, we endow
$M^{\vee}$ with the contragredient $\mathcal{G}$-action.

\begin{theorem} \label{thm:iso-XY}
	Put $n := [K:\Q_p]$. Then for every $\mathfrak p \in \boldsymbol{P}_0$ 
	the following hold.
	\begin{enumerate}
		\item 
		If $\mu_{p}(L_{\infty})$ is finite, then
		we have isomorphisms of $\Lambda_{\mathfrak p}(\mathcal{G})$-modules
		\[
		(Y_{L_{\infty}})_{\mathfrak p} \simeq 
		\Lambda_{\mathfrak p}(\mathcal{G})^{n+1}, \quad
		(X_{L_{\infty}})_{\mathfrak p} \simeq 
		\Lambda_{\mathfrak p}(\mathcal{G})^n.
		\]
		\item 
		If $\mu_{p}(L_{\infty})$ is infinite (and thus
		$L_{\infty}/L$ is the cyclotomic $\Z_p$-extension), then
		we have isomorphisms of $\Lambda_{\mathfrak p}(\mathcal{G})$-modules
		\[
		(Y_{L_{\infty}})_{\mathfrak p} \simeq (\Z_p(1))_{\mathfrak p} \oplus
		\Lambda_{\mathfrak p}(\mathcal{G})^{n+1}, \quad
		(X_{L_{\infty}})_{\mathfrak p} \simeq (\Z_p(1))_{\mathfrak p} \oplus
		\Lambda_{\mathfrak p}(\mathcal{G})^n.
		\]
	\end{enumerate}
\end{theorem}

\begin{proof}
	We first note that it suffices to prove the result
	for $(Y_{L_{\infty}})_{\mathfrak p}$.
	As on earlier occasions, we may then use \cite[Proposition 30.17
	and Corollary 6.15]{MR632548} to deduce the result for
	$(X_{L_{\infty}})_{\mathfrak p}$ from Lemma \ref{lem:first-prop}(i).
	
	As the $p$-dualizing module of $G_K$ naturally identifies with
	$\Q_p/\Z_p(1)$ by \cite[Theorem 7.2.4]{MR2392026}, we have a
	homotopy equivalence of $\Lambda(\mathcal{G})$-modules
	\begin{equation} \label{eqn:basic-homotopy}
		Y_{L_{\infty}} \sim D(\mu_{p}(L_{\infty})^{\vee})
	\end{equation}
	by \cite[Proposition 5.6.9]{MR2392026}.
	We first assume that $\mu_{p}(L_{\infty})$ is finite. Then 
	\eqref{eqn:basic-homotopy} implies that 
	$(Y_{L_{\infty}})_{\mathfrak p} \sim 0$.
	This means that $(Y_{L_{\infty}})_{\mathfrak p}$ is a projective
	$\Lambda_{\mathfrak p}(\mathcal{G})$-module. As $\mathcal{Q}(\mathcal{G})$ is semisimple, the 
	$\mathcal{Q}(\mathcal{G})$-module
	$\mathcal{Q}(\mathcal{G}) \otimes_{\Lambda(\mathcal{G})} Y_{L_{\infty}}$
	is free of rank $n+1$ by
	\cite[Theorem 7.4.2]{MR2392026}.  
	Corollary \ref{cor:Swan-Iwasawa} then gives the result.	
	
	Now assume that $\mu_{p}(L_{\infty})$ is infinite. Then
	\eqref{eqn:basic-homotopy} gives
	\[
		Y_{L_{\infty}} \sim D(\Z_p(-1)).
	\]
	As the functor $D$ induces an autoduality, 
	we have $E^1(DY_{L_{\infty}}) = 
	E^1(\Z_p(-1)) = \Z_p(1)$ and likewise $E^2(DY_{L_{\infty}})
	= E^2(\Z_p(-1)) = 0$ by \cite[Proposition 5.5.3 (iv)
	and Corollary 5.5.7]{MR2392026}. 
	Thus sequence \eqref{eqn:transpose-sequence}
	specializes to
	\begin{equation} \label{eqn:transpose-ses-Y}
	0 \longrightarrow \Z_p(1) \longrightarrow Y_{L_{\infty}}
	\longrightarrow Y_{L_{\infty}}^{++} \longrightarrow 0.
	\end{equation}
	Since $(\Z_p(1))_{(p)}$ vanishes,
	we have $\pd_{\Lambda_{\mathfrak p}(\mathcal{G})} 
	((\Z_p(1))_{\mathfrak p}) \leq 1$ for every $\mathfrak p \in
	\boldsymbol{P}_0$ by Corollary \ref{cor:pd1}.
	The projective dimension of $(Y_{L_{\infty}})_{\mathfrak p}$
	is at most $1$ by Lemma \ref{lem:first-prop}(ii)
	and thus $(Y_{L_{\infty}}^{++})_{\mathfrak p}$
	also has finite projective dimension.
	Lemma \ref{lem:M+projective} implies that the 
	$\Lambda_{\mathfrak p}(\mathcal{G})$-module
	$(Y_{L_{\infty}}^{++})_{\mathfrak p}$ is indeed projective.
	We now may deduce from Corollary \ref{cor:Swan-Iwasawa}
	as above that $(Y_{L_{\infty}}^{++})_{\mathfrak p}$ is free of rank
	$n+1$. By \eqref{eqn:transpose-ses-Y} we get an isomorphism
	\[
		(Y_{L_{\infty}})_{\mathfrak p} \simeq (\Z_p(1))_{\mathfrak p}
		\oplus \Lambda_{\mathfrak p}(\mathcal{G})^{n+1}
	\]
	as desired.
\end{proof}

\begin{corollary} \label{cor:iso-U1}
	For every $\mathfrak p \in \boldsymbol{P}_0$ 
	we have an isomorphism of $\Lambda_{\mathfrak p}(\mathcal{G})$-modules
	\[
	U^1(L_{\infty})_{\mathfrak p} \simeq
	(X_{L_{\infty}})_{\mathfrak p}.
	\]
	In particular, the following hold.
	\begin{enumerate}
		\item 
		If $\mu_{p}(L_{\infty})$ is finite, then
		we have an isomorphism of $\Lambda_{\mathfrak p}(\mathcal{G})$-modules
		\[
		U^1(L_{\infty})_{\mathfrak p} \simeq 
		\Lambda_{\mathfrak p}(\mathcal{G})^n.
		\]
		\item 
		If $\mu_{p}(L_{\infty})$ is infinite (and thus
		$L_{\infty}/L$ is the cyclotomic $\Z_p$-extension), then
		we have an isomorphism of $\Lambda_{\mathfrak p}(\mathcal{G})$-modules
		\[
		U^1(L_{\infty})_{\mathfrak p} \simeq (\Z_p(1))_{\mathfrak p} \oplus
		\Lambda_{\mathfrak p}(\mathcal{G})^n.
		\]
	\end{enumerate}
\end{corollary}

\begin{proof}
	If $L_{\infty}$ is the unramified $\Z_p$-extension, then
	$U^1(L_{\infty}) \simeq X_{L_{\infty}}$ and the result immediately
	follows from Theorem \ref{thm:iso-XY}. Now suppose that 
	$L_{\infty}/L$ is ramified. Let us put
	$Z_{\mathfrak p^{\sharp}} := E^1((X_{L_{\infty}})_{\mathfrak p})$. Then
	Theorem \ref{thm:iso-XY} implies that
	$Z_{\mathfrak p^{\sharp}}$ vanishes unless 
	$\mu_p(L_{\infty})$ is infinite,
	where we have an isomorphism of 
	$\Lambda_{\mathfrak p^{\sharp}}(\mathcal{G})$-modules
	$Z_{\mathfrak p^{\sharp}} \simeq \Z_p(-1)_{\mathfrak p^{\sharp}}$.
	In both cases we have that
	\begin{equation} \label{eqn:Hom-vanishes}
		\Hom_{\Lambda_{\mathfrak p^{\sharp}}(\mathcal{G})} 
		((\Z_p)_{\mathfrak p^{\sharp}},
		Z_{\mathfrak p^{\sharp}}) = 0.
	\end{equation}
	The exact sequence \eqref{eqn:ramified-ses} 
	localized at $\mathfrak p$ induces a long exact sequence of 
	$\Lambda_{\mathfrak p^{\sharp}}(\mathcal{G})$-modules
	\[
	\cdots \longrightarrow E^1((\Z_p)_{\mathfrak p})
	\longrightarrow E^1((X_{L_{\infty}})_{\mathfrak p})
	\longrightarrow E^1(U^1(L_{\infty})_{\mathfrak p})
	\longrightarrow E^2((\Z_p)_{\mathfrak p})
	\longrightarrow \cdots
	\]
	As we have an isomorphism 
	$E^1((\Z_p)_{\mathfrak p}) \simeq (\Z_p)_{\mathfrak p^{\sharp}}$,
	the second arrow is the zero map by \eqref{eqn:Hom-vanishes}.
	Since $E^i((\Z_p)_{\mathfrak p})$ vanishes for $i \not=1$, 
	we obtain an isomorphism of 
	$\Lambda_{\mathfrak p^{\sharp}}(\mathcal{G})$-modules
	\[
		E^1(U^1(L_{\infty})_{\mathfrak p}) \simeq
		E^1((X_{L_{\infty}})_{\mathfrak p}) = Z_{\mathfrak p^{\sharp}}.
	\]
	In particular
	$E^1(E^1(U^1(L_{\infty})_{\mathfrak p})) \simeq
	E^1(Z_{\mathfrak p^{\sharp}})$ vanishes 
	unless $\mu_p(L_{\infty})$ is infinite,
	where we have $E^1(Z_{\mathfrak p^{\sharp}}) \simeq 
	\Z_p(1)_{\mathfrak p}$.
	Now \cite[Proposition 5.5.8]{MR2392026} and \eqref{eqn:transpose-sequence}
	imply that 
	\[
		E^1(DU^1(L_{\infty})_{\mathfrak p}) \simeq 
		E^1(Z_{\mathfrak p^{\sharp}})
	\]
	which in particular has projective dimension at most $1$.
	It follows that $U^1(L_{\infty})_{\mathfrak p}$ and
	$U^1(L_{\infty})^{++}_{\mathfrak p}$ have finite projective dimension
	by \eqref{eqn:ramified-ses}, Lemma \ref{lem:first-prop} and the
	exact sequence
	\begin{equation} \label{eqn:U1-ses}
		0 \longrightarrow E^1(Z_{\mathfrak p^{\sharp}}) \longrightarrow
		U^1(L_{\infty})_{\mathfrak p} \longrightarrow
		U^1(L_{\infty})^{++}_{\mathfrak p} \longrightarrow 0.
	\end{equation}
	Thus $U^1(L_{\infty})^{++}_{\mathfrak p}$ is indeed a projective
	$\Lambda_{\mathfrak p}(\mathcal{G})$-module by 
	Lemma \ref{lem:M+projective}. We have
	\[
		\mathcal{Q}(\mathcal{G}) \otimes_{\Lambda(\mathcal{G})} 
		U^1(L_{\infty})^{++} \simeq
		\mathcal{Q}(\mathcal{G}) \otimes_{\Lambda(\mathcal{G})} 
		U^1(L_{\infty}) \simeq
		\mathcal{Q}(\mathcal{G}) \otimes_{\Lambda(\mathcal{G})} 
		X_{L_{\infty}} \simeq
		\mathcal{Q}(\mathcal{G})^n
	\]
	by Theorem \ref{thm:iso-XY}.
	It now follows from Corollary \ref{cor:Swan-Iwasawa} that $U^1(L_{\infty})^{++}_{\mathfrak p}$
	is a free $\Lambda_{\mathfrak p}(\mathcal{G})$-module of rank $n$.
	Sequence \eqref{eqn:U1-ses} splits, giving the claim.
\end{proof}

\subsection{Iwasawa theory of $\ell$-adic fields}
We briefly discuss the case where $L/K$ is a finite Galois extension
of $\ell$-adic fields with $p \not= \ell$. Then $L$ has a unique 
$\Z_p$-extension $L_{\infty}$, namely the unramified $\Z_p$-extension.
We again define $\mathcal{G} := \Gal(L_{\infty}/K)$.
For each $n \geq 0$, the valuation map induces an exact sequence
\[
	0 \longrightarrow \mu_p(L_n) \longrightarrow \widehat{L_n^{\times}}
	\longrightarrow \Z_p \longrightarrow 0.
\]
Taking inverse limits over all $n$ yields an isomorphism
of $\Lambda(\mathcal{G})$-modules
\[
	\varprojlim_n \mu_p(L_n) \simeq	
	\varprojlim_n \widehat{L_n^{\times}} =: X_{L_{\infty}}.
\]
The following result is therefore clear 
(see also \cite[Theorem 11.2.3(ii)]{MR2392026}).
We let $\zeta_p$ be a primitive $p$-th roof of unity.

\begin{lemma} \label{lem:l-adic-case}
	For $\ell \not= p$ we have $X_{L_{\infty}} \simeq \Z_p(1)$ if
	$\zeta_p \in L$ and $X_{L_{\infty}} = 0$ otherwise.
\end{lemma}

\section{Iwasawa theory of number fields} \label{sec:global}

\subsection{The relevant Galois groups}
In this section we consider a finite Galois extension $L/K$ of 
number fields with Galois group $G$. Let $p$ be a prime and let
$L_{\infty}$ be the cyclotomic $\Z_p$-extension of $L$ with $n$-th
layer $L_n$.
We will assume throughout that 
\begin{center}
	\emph{$K$ is totally imaginary if $p=2$}.
\end{center}
We put $\mathcal{G} := \Gal(L_{\infty}/K)$ which is a one-dimensional
$p$-adic Lie group. We may write $\mathcal{G} \simeq H \rtimes \Gamma$,
where $H$ naturally identifies with a subgroup of $G$ 
and $\Gamma \simeq \Z_p$. For every place $v$ of $K$ we choose a place
$w_{\infty}$ of $L_{\infty}$ above $v$ and let $\mathcal{G}_v$
be the decomposition group at $w_{\infty}$. We denote the place
of $L$ below $w_{\infty}$ by $w$ and the completion of $L$ at $w$
by $L_w$.

We choose a finite set $S$ of places of $K$ containing all archimedean
places and all places that ramify in $L_{\infty}/K$. In particular,
all $p$-adic places lie in $S$. 
We denote the ring of integers in $L$ by $\mathcal{O}_L$ and the
ring of $S(L)$-integers by $\mathcal{O}_{L,S}$,
where $S(L)$ denotes the set of places of $L$ above those in $S$.

We let $M_S$ be the maximal pro-$p$-extension
of $L$ which is unramified outside $S$. We put 
$G_S := \Gal(M_S/K)$ and $H_S := \Gal(M_S/L_{\infty})$.
Since $K$ is totally imaginary
if $p=2$, the cohomological $p$-dimension of $G_S$ equals
$2$ by \cite[Proposition 10.11.3]{MR2392026}
(note that our definition of $G_S$ follows
\cite[Chapter XI, \S 3, p.739]{MR2392026}, but slightly
differs from the profinite group $G_S$ considered
in \cite[Chapter X, \S 11]{MR2392026}; however,
the proof of \cite[Lemma 5.3]{MR1097615} shows that both
groups have the same cohomological $p$-dimension).
Choose a presentation $F_d \twoheadrightarrow G_S$ of $G_S$
by a free profinite group $F_d$ of finite rank $d$. Then we obtain a 
commutative diagram (compare \cite[p.~740]{MR2392026})
\[ \xymatrix{
	& 1 \ar[d] & 1 \ar[d] & & \\
	& N \ar@{=}[r] \ar[d] & N \ar[d] & & \\
	1 \ar[r] & R \ar[r] \ar[d] & F_d \ar[r] \ar[d] & 
		\mathcal{G} \ar[r] \ar@{=}[d] & 1 \\
	1 \ar[r] & H_S \ar[r] \ar[d] & G_S \ar[r] \ar[d] & 
		\mathcal{G} \ar[r] & 1 \\
	& 1 & 1 & & 
}\]
with exact rows and columns, where $R$ and $N$ are the kernels of
$F_d \twoheadrightarrow \mathcal{G}$ and $F_d \twoheadrightarrow G_S$,
respectively. Then $G_S$ acts on $N^{\ab}(p)$, the maximal abelian
pro-$p$-quotient of $N$. The module $N^{\ab}_{H_S}(p)$ of $H_S$-coinvariants
of $N^{\ab}(p)$ is a projective $\Lambda(\mathcal{G})$-module
by \cite[Proposition 5.6.7]{MR2392026}. We let $r_1$ and $r_2$ be the
number of real and complex places of $K$, respectively. We let
$S_{\infty}'$ be the set of real places of $K$ becoming complex
in $L_{\infty}$ and put $r_1' := |S_{\infty}'|$.
If we choose $d$ greater than or equal to $r_2 + r_1' + 1$, then
we have an isomorphism of $\Lambda(\mathcal{G})$-modules
\begin{equation} \label{eqn:NabHSp}
	N^{\ab}_{H_S}(p) \simeq
	\Lambda(\mathcal{G})^{d - r_2 - r_1' - 1} \oplus
	\bigoplus_{v \in S_{\infty}'} \Ind_{\mathcal{G}_v}^{\mathcal{G}} \Z_p
\end{equation}
by \cite[Theorem 5.4]{MR1097615}
(see also \cite[Theorem 11.3.10(iii)]{MR2392026};
the assumption that $p$ does not divide
$[L:K]$ is not needed for this part of the theorem).
Here, for a closed subgroup $\mathcal{H}$ of $\mathcal{G}$
and a compact $\Lambda(\mathcal{H})$-module $M$ we let
\[
	\Ind_{\mathcal{H}}^{\mathcal{G}} M := \Lambda(\mathcal{G})
	\widehat{\otimes}_{\Lambda(\mathcal{H})} M
\]
denote compact induction of $M$ from $\mathcal{H}$ to $\mathcal{G}$.

\subsection{Global and semi-local Iwasawa modules}
Let $X_S := H_S^{\ab}$ be the abelianization of $H_S$.
Then $X_S$ is a finitely generated $\Lambda(\mathcal{G})$-module
by \cite[Proposition 11.3.1]{MR2392026}. 
We also consider the `standard' Iwasawa module $X_{nr}$, which is the
Galois group over $L_{\infty}$ of the maximal unramified abelian 
pro-$p$-extension of $L_{\infty}$, and the quotient $X_{cs}^S$ of
$X_{nr}$ that corresponds to the maximal subextension which is completely
decomposed at all primes above $S$. For a finite place $v$ of $K$
we define
\[
	A_v := \varprojlim_n \prod_{w_n \mid v} \widehat{L_{n,w_n}^{\times}}
	\simeq \Ind_{\mathcal{G}_v}^{\mathcal{G}} X_{L_{w, \infty}},
\]
\[
U_v := \varprojlim_n \prod_{w_n \mid v} \widehat{\mathcal{O}_{L_{n,w_n}}^{\times}}
\simeq \left\{ \begin{array}{lll}
\Ind_{\mathcal{G}_v}^{\mathcal{G}} X_{L_{w, \infty}} 
	& \mbox{ if } & v \nmid p\\
\Ind_{\mathcal{G}_v}^{\mathcal{G}} U^1(L_{w, \infty}) 
& \mbox{ if } & v \mid p.
\end{array}
\right.
\]
Here, $L_{w,\infty}$ always denotes the cyclotomic $\Z_p$-extension
of $L_{w}$.
We let $S_f$ be the subset of $S$ comprising all finite places in $S$.
We then define $\Lambda(\mathcal{G})$-modules
\[
A_S := \prod_{v \in S_f} A_v, \quad
U_S := \prod_{v \in S_f} U_v.
\]
Finally, we let
\[
	E_S := \varprojlim_n (\mathcal{O}_{L_n,S}^{\times} \otimes_{\Z} \Z_p),
	 \quad
	E := \varprojlim_n (\mathcal{O}_{L_n}^{\times} \otimes_{\Z} \Z_p).
\]
Since the weak Leopoldt conjecture holds for the cyclotomic $\Z_p$-extension
by \cite[Theorem 10.3.25]{MR2392026}, we obtain from
\cite[Theorem 5.4]{MR1097615} the following commutative
diagram of $\Lambda(\mathcal{G})$-modules with exact rows (see also
\cite[Theorem 11.3.10(i)]{MR2392026}; the assumption 
$p \nmid [L:K]$ is irrelevant, since all maps are certainly
$\mathcal{G}$-equivariant)
\begin{equation} \label{eqn:diagram}
\xymatrix{
	0 \ar[r] & E \ar[r] \ar@{^{(}->}[d] & U_S \ar[r] \ar@{^{(}->}[d] 
		& X_S \ar[r] \ar@{=}[d] & X_{nr} \ar[r] \ar@{->>}[d] & 0 \\
	0 \ar[r] & E_S \ar[r] & A_S \ar[r] & X_S \ar[r] & X_{cs}^S \ar[r] & 0.
	}
\end{equation}

As in the local case, there is an exact sequence of 
$\Lambda(\mathcal{G})$-modules (see \cite[Proposition 5.6.7]{MR2392026})
\begin{equation} \label{eqn:XYDelta}
	0 \longrightarrow X_S \longrightarrow Y_S \longrightarrow \Delta(\mathcal{G})
	\longrightarrow 0,
\end{equation}
where $Y_S := \Delta(G_S)_{H_S}$ is a finitely generated
$\Lambda(\mathcal{G})$-module of projective dimension at most $1$.

\subsection{Structure of global Iwasawa modules}
We now determine the structure of the above Iwasawa modules 
after localization at a prime ideal $\mathfrak p \in \boldsymbol{P}_0$.
We begin with the semi-local Iwasawa modules.

\begin{prop} \label{prop:semi-local}
	Let $S_f(\zeta_p)$ be the set of all finite places $v$ in $S$
	such that $\zeta_p \in L_w$ and put $n := [K:\Q]$. Then for every
	$\mathfrak p \in \boldsymbol{P}_0$ we have isomorphisms of
	$\Lambda_{\mathfrak p}(\mathcal{G})$-modules
	\[
		(A_S)_{\mathfrak p} \simeq (U_S)_{\mathfrak p} \simeq
		\Lambda_{\mathfrak p}(\mathcal{G})^n \oplus
		\bigoplus_{v \in S_f(\zeta_p)} 
		\left(\Ind_{\mathcal{G}_v}^{\mathcal{G}} \Z_p(1)\right)_{\mathfrak p}.
	\]
	In particular, we have 
	$\pd_{\Lambda_{\mathfrak p}(\mathcal{G})}((A_S)_{\mathfrak p}) = 
	\pd_{\Lambda_{\mathfrak p}(\mathcal{G})}((U_S)_{\mathfrak p}) \leq 1$.
\end{prop}

\begin{proof}
	This follows from Theorem \ref{thm:iso-XY}, Corollary \ref{cor:iso-U1},
	Lemma \ref{lem:l-adic-case} and the well known formula 
	$[K:\Q] = \sum_{v \mid p} [K_v : \Q_p]$.
\end{proof}

We let $D_2^{(p)}(G_S)$ be the $p$-dualizing module of $G_S$ and put
$Z_S := (D_2^{(p)}(G_S)^{H_S})^{\vee}$.

\begin{lemma} \label{lem:first-prop-global}
	For every $\mathfrak p \in \boldsymbol{P}_0$ the following hold.
	\begin{enumerate}
		\item 
		We have an isomorphism of $\Lambda_{\mathfrak p}(\mathcal{G})$-modules
		\[
			(Y_S)_{\mathfrak p} \simeq (X_S)_{\mathfrak p} \oplus 
			\Lambda_{\mathfrak p}(\mathcal{G});
		\]
		\item
		we have 
		$\pd_{\Lambda_{\mathfrak p}(\mathcal{G})}((X_S)_{\mathfrak p}) = 
		\pd_{\Lambda_{\mathfrak p}(\mathcal{G})}((Y_S)_{\mathfrak p}) \leq 1$;
		\item
		we have a homotopy equivalence 
		$(X_S)_{\mathfrak p} \sim (DZ_S)_{\mathfrak p}$ 
		and an isomorphism of 
		$\Lambda_{\mathfrak p^{\sharp}}(\mathcal{G})$-modules
		\[
			E^1((X_S)_{\mathfrak{p}}) \simeq (Z_S)_{\mathfrak p^{\sharp}}.
		\]
	\end{enumerate}
\end{lemma}

\begin{proof}
	Sequence \eqref{eqn:XYDelta} and Corollary \ref{cor:augideal-free}
	imply (i). As $Y_S$ is a $\Lambda(\mathcal{G})$-module of
	projective dimension at most $1$, (i) implies (ii).
	By (i) we have $(X_S)_{\mathfrak p} \sim
	(Y_S)_{\mathfrak p}$ and  in particular 
	$E^1((X_S)_{\mathfrak{p}}) \simeq
	E^1((Y_S)_{\mathfrak{p}})$. Hence (iii) is a consequence
	of \cite[Proposition 5.6.9]{MR2392026}.
\end{proof}

We let $\mu_L$ be the Iwasawa $\mu$-invariant of the standard
Iwasawa module $X_{nr}$. We recall the following conjecture of Iwasawa.

\begin{conj}[Iwasawa]
	For every number field $L$ the $\mu$-invariant $\mu_L$ vanishes.
\end{conj}

The following two results are analogues of \cite[Theorem 11.3.11]{MR2392026}
for arbitrary one-dimensional $p$-adic Lie extensions (containing the
cyclotomic $\Z_p$-extension).

\begin{theorem} \label{thm:iso-X_S}
	Let $\mathfrak p \in \boldsymbol{P}_0$ and assume that 
	$\mu_{L(\zeta_p)} = 0$ if $\mathfrak p = (p)$. Then the following hold.
	\begin{enumerate}
		\item 
		We have an isomorphism of $\Lambda_{\mathfrak p}(\mathcal{G})$-modules
		\[
			(X_S)_{\mathfrak p} \simeq E^1(Z_S)_{\mathfrak p} \oplus
			(X_S)_{\mathfrak p}^{++};
		\]
		moreover, we have $(Z_S)_{(p)} = 0$ so that in particular
		$(X_S)_{(p)}  \simeq  (X_S)_{(p)}^{++}$;
		\item
		we have isomorphisms of 
		$\Lambda_{\mathfrak p}(\mathcal{G})$-modules
		\[
			((X_S)_{\mathfrak p^{\sharp}})^{+} \simeq 
			(X_S)_{\mathfrak p}^{++} \simeq
			\Lambda_{\mathfrak p}(\mathcal{G})^{r_2} \oplus
			\bigoplus_{v \in S_{\infty}'} 
			\left(\Ind_{\mathcal{G}_v}^{\mathcal{G}} \Z_p^-
			\right)_{\mathfrak p},
		\]
		where $\Z_p^-$ is the $\mathcal{G}_v$-module $\Z_p$ upon which
		the generator of $\mathcal{G}_v \simeq \Z/ 2\Z$ acts by multiplication
		by $-1$.
	\end{enumerate}
\end{theorem}

\begin{proof}
	By Lemma \ref{lem:first-prop-global}(iii) we have 
	$E^1(D X_S)_{\mathfrak p} \simeq E^1(Z_S)_{\mathfrak p}$ so that
	(i) follows from Corollary \ref{cor:M-splits} if $\mathfrak p \not= (p)$.
	We claim that $(Z_S)_{(p)}$ vanishes. Then Lemma \ref{lem:E2-vanishes}
	implies (i) in the case $\mathfrak p = (p)$.
	We first assume that $\zeta_p \in L$. Then
	by \cite[Theorem 5.4 (d)]{MR1097615} there is an exact sequence
	of $\Lambda(\Gamma_0)$-modules
	\[
		0 \longrightarrow X_{cs}^S(-1) \longrightarrow Z_S \longrightarrow T
		\longrightarrow 0,
	\]
	where $T$ is finitely generated and free as $\Z_p$-module.
	As $\mu_L$ vanishes by assumption and $X_{nr}$ surjects onto
	$X_{cs}^S$, the latter module is also finitely generated over $\Z_p$.
	Hence the same is true for $Z_S$ and so $(Z_S)_{(p)} = 0$
	as desired. If $\zeta_p$ is not in $L$, we put $L' := L(\zeta_p)$
	and $\Delta := \Gal(L'/L) \simeq \Gal(L'_{\infty}/ L_{\infty})$.
	Let $Z_S'$ be the Iwasawa module $Z_S$ that corresponds to $L'$.
	We have shown that $Z_S'$ is a finitely generated $\Z_p$-module.
	However, there is a natural isomorphism $(Z_S')_{\Delta} \simeq Z_S$
	so that the $\mu$-invariant of $Z_S$ also vanishes. This proves the
	claim and thus (i). Lemmas \ref{lem:M+projective},
	\ref{lem:reflexive-projective} and 
	\ref{lem:first-prop-global}(ii) imply that both
	$((X_S)_{\mathfrak p^{\sharp}})^{+}$ and $(X_S)_{\mathfrak p}^{++}$ are
	projective $\Lambda_{\mathfrak p}(\mathcal{G})$-modules.
	By Corollary \ref{cor:Swan-Iwasawa} it now suffices to compute
	$\mathcal{Q}(\mathcal{G}) \otimes_{\Lambda(\mathcal{G})} X_S^{++}$.
	By \cite[Proposition 5.6.7]{MR2392026} we have
	\[
		\mathcal{Q}(\mathcal{G}) \otimes_{\Lambda(\mathcal{G})} (X_S^{++} \oplus N^{\ab}_{H_S}(p)) =  
		\mathcal{Q}(\mathcal{G}) \otimes_{\Lambda(\mathcal{G})} (X_S
		\oplus N^{\ab}_{H_S}(p)) \simeq
		\mathcal{Q}(\mathcal{G})^{d-1}.
	\]
	Since $\mathcal{Q}(\mathcal{G})$ is semisimple, (ii) now follows
	from \eqref{eqn:NabHSp}.
\end{proof}

\begin{theorem} \label{thm:iso-E_S}
	Let $\mathfrak p \in \boldsymbol{P}_0$ and assume that 
	$\mu_{L(\zeta_p)} = 0$ if $\mathfrak p = (p)$. Then the following hold.
	\begin{enumerate}
		\item
		We have $\pd_{\Lambda_{\mathfrak p}(\mathcal{G})}(E_{\mathfrak p}) = 
		\pd_{\Lambda_{\mathfrak p}(\mathcal{G})}((E_S)_{\mathfrak p}) \leq 1$;
		\item 
		if $\zeta_p \in L$ we have isomorphisms of 
		$\Lambda_{\mathfrak p}(\mathcal{G})$-modules
		\[
			E_{\mathfrak p} \simeq (E_S)_{\mathfrak p}  \simeq   
			\Z_p(1)_{\mathfrak p} \oplus
			\Lambda_{\mathfrak p}(\mathcal{G})^{r_2 + r_1 - r_1'} \oplus
			\bigoplus_{v \in S_{\infty}'} 
			\left(\Ind_{\mathcal{G}_v}^{\mathcal{G}} \Z_p
			\right)_{\mathfrak p};
		\]
		\item
		if $\zeta_p \not\in L$ we have isomorphisms of 
		$\Lambda_{\mathfrak p}(\mathcal{G})$-modules
		\[
		E_{\mathfrak p} \simeq (E_S)_{\mathfrak p}  \simeq   
		\Lambda_{\mathfrak p}(\mathcal{G})^{r_2 + r_1 - r_1'} \oplus
		\bigoplus_{v \in S_{\infty}'} 
		\left(\Ind_{\mathcal{G}_v}^{\mathcal{G}} \Z_p\right)_{\mathfrak p}.
		\]
	\end{enumerate}
\end{theorem}

\begin{proof}
	We first show that the projective dimension of
	$E_{\mathfrak p}$ and $(E_S)_{\mathfrak p}$ is at most $1$.
	For this we only have to treat the case $\mathfrak p = (p)$.
	Otherwise we apply Corollary \ref{cor:pd1}. As the $\mu$-invariant
	of $X_{nr}$ vanishes by assumption, we obtain from diagram
	\eqref{eqn:diagram} two exact sequences of
	$\Lambda_{(p)}(\mathcal{G})$-modules
	\[
		0 \longrightarrow E_{(p)} \longrightarrow (U_S)_{(p)} \longrightarrow
		(X_S)_{(p)} \longrightarrow 0,
	\]
	\[
	0 \longrightarrow (E_S)_{(p)} \longrightarrow (A_S)_{(p)} \longrightarrow
	(X_S)_{(p)} \longrightarrow 0.
	\]
	Since the projective dimension of $(U_S)_{(p)}$, $(A_S)_{(p)}$
	and $(X_S)_{(p)}$ is at most $1$ by
	Proposition \ref{prop:semi-local} and Lemma
	\ref{lem:first-prop-global}(ii), 
	the same is true for
	$E_{(p)}$ and $(E_S)_{(p)}$.
	
	Now let $\mathfrak p \in \boldsymbol{P}_0$ be arbitrary.
	It follows as in the proof of \cite[Theorem 11.3.11(ii)]{MR2392026}
	that $E^1(D(E_S)_{\mathfrak p}) \simeq \Z_p(1)_{\mathfrak p}$
	if $\zeta_p \in L$ and that $E^1(D(E_S)_{\mathfrak p})$ vanishes
	otherwise. In both cases we have
	$\pd_{\Lambda_{\mathfrak p}(\mathcal{G})}(E^1(D(E_S)_{\mathfrak p})) 
	\leq 1$ and thus $(E_S)_{\mathfrak p}^{++}$ is projective
	by Lemma \ref{lem:M+projective}. It follows that $(E_S)_{\mathfrak p}$
	decomposes into a direct sum
	\[
		(E_S)_{\mathfrak p} \simeq E^1(D(E_S)_{\mathfrak p}) \oplus
		(E_S)_{\mathfrak p}^{++}.
	\]
	The inclusions $E^1(DE_{\mathfrak p}) \subseteq E^1(D(E_S)_{\mathfrak p})
	\subseteq E_{\mathfrak p}$ imply that in fact $E^1(DE_{\mathfrak p}) =
	E^1(D(E_S)_{\mathfrak p})$. It follows as above that the module
	$E_{\mathfrak p}^{++}$ is projective and that we have an isomorphism
	\[
	E_{\mathfrak p} \simeq E^1(DE_{\mathfrak p}) \oplus
	E_{\mathfrak p}^{++}.
	\]
	In particular, we obtain (i). By Corollary \ref{cor:Swan-Iwasawa}
	it now suffices to compute
	\[
		\mathcal{Q}(\mathcal{G}) \otimes_{\Lambda(\mathcal{G})} E_S^{++}
		= \mathcal{Q}(\mathcal{G}) \otimes_{\Lambda(\mathcal{G})} E_S
		= \mathcal{Q}(\mathcal{G}) \otimes_{\Lambda(\mathcal{G})} E
		= \mathcal{Q}(\mathcal{G}) \otimes_{\Lambda(\mathcal{G})} E^{++}.
	\]
	We deduce from diagram \eqref{eqn:diagram} 
	and Proposition \ref{prop:semi-local}
	that we have isomorphisms
	of $\mathcal{Q}(\mathcal{G})$-modules
	\[
		\mathcal{Q}(\mathcal{G}) \otimes_{\Lambda(\mathcal{G})}
		(E_S \oplus X_S) \simeq 
		\mathcal{Q}(\mathcal{G}) \otimes_{\Lambda(\mathcal{G})} A_S
		\simeq \mathcal{Q}(\mathcal{G})^n.
	\]
	As $\mathcal{Q}(\mathcal{G})$ is semisimple, the result follows
	from Theorem \ref{thm:iso-X_S}.
\end{proof}

\begin{remark}
	Let $L_{\infty}$ be an arbitrary $\Z_p$-extension of $L$ such that
	$L_{\infty}/K$ is again a Galois extension. Assuming the validity
	of the weak Leopoldt conjecture, it seems to be likely that one can prove
	analogues of Theorems \ref{thm:iso-X_S} and \ref{thm:iso-E_S}.
	The main obstacle occurs in the case $\mathfrak p = (p)$
	because the relevant $\mu$-invariant does not vanish in general.
\end{remark}

\nocite*
\bibliography{swan-iwasawa-bib}{}
\bibliographystyle{amsalpha}

\end{document}